\documentclass[12pt]{amsart}
\usepackage[latin1]{inputenc}
\usepackage{mathptmx}
\usepackage{amscd}
\usepackage{amssymb}
\newtheorem{theorem}{Theorem}
\newtheorem{lemma}[theorem]{Lemma}
\newtheorem{corollary}[theorem]{Corollary}

\theoremstyle{definition}
\newtheorem{remark}[theorem]{Remark}
\newtheorem{definition}[theorem]{Definition}

\newtheorem{example}[theorem]{Example}

\let\phi=\varphi

\let\oldbigwedge\bigwedge
\def\BIGwedge{{\textstyle\oldbigwedge}}
\def\medwedge{{\scriptstyle\oldbigwedge}}
\def\bigwedge{\mathchoice{\BIGwedge}{\BIGwedge}{\medwedge}{}}

\let\epsilon=\varepsilon

\begin{document}

\title{On the Anderson-Badawi $\omega_{R[X]}(I[X])=\omega_R(I)$ conjecture}

\author{Peyman Nasehpour}

\address{Peyman Nasehpour, Department of Engineering Science, Faculty of Engineering, University of Tehran, Tehran, Iran}
\email{nasehpour@gmail.com}

\begin{abstract}
Let $R$ be a commutative ring with an identity different from zero and $n$ be a positive integer. Anderson and Badawi, in their paper on $n$-absorbing ideals, define a proper ideal $I$ of a commutative ring $R$ to be an $n$-absorbing ideal of $R$, if whenever $x_1 \cdots x_{n+1} \in I$ for $x_1, \ldots, x_{n+1} \in R$, then there are $n$ of the $x_i$'s whose product is in $I$ and conjecture that $\omega_{R[X]}(I[X])=\omega_R(I)$ for any ideal $I$ of an arbitrary ring $R$, where $\omega_R(I)= \min \{n\colon\text{$I$ is an $n$-absorbing ideal of $R$}\}$. In the present paper, we use content formula techniques to prove that their conjecture is true, if one of the following conditions hold:

\begin{enumerate}

\item The ring $R$ is a Pr\"{u}fer domain.

\item The ring $R$ is a Gaussian ring such that its additive group is torsion-free.

\item The additive group of the ring $R$ is torsion-free and $I$ is a radical ideal of $R$.

\end{enumerate}
\end{abstract}

\maketitle

\section{Introduction}

Let $R$ be a commutative ring with an identity different from zero and $n$ be a positive integer. Anderson and Badawi, in their paper \cite{AB}, define a proper ideal $I$ of a commutative ring $R$ to be an $n$-absorbing ideal of $R$, if whenever $x_1 \cdots x_{n+1} \in I$ for $x_1, \ldots, x_{n+1} \in R$, then there are $n$ of the $x_i$'s whose product is in $I$. In the fourth section of their paper, they conjecture that $\omega_{R[X]}(I[X])=\omega_R(I)$ for any ideal $I$ of an arbitrary ring $R$, where $\omega_R(I)= \min \{n \colon \text{ $I$ is an $n$-absorbing ideal of $R$}\}$\let\thefootnote\relax\footnote{2010 Mathematics Subject Classification: Primary 13A15; Secondary 13B02, 13B25, 13F05.
Keywords: $n$-absorbing ideals, strongly $n$-absorbing ideals, polynomial rings, content algebras, Dedekind-Mertens content formula, Pr\"{u}fer domains, Gaussian algebras, Gaussian rings}.

Clearly a $1$-absorbing ideal is just a prime ideal and it is a well-known result in commutative ring theory that $I$ is a prime ideal of $R$ iff $I[X]$ is a prime ideal of $R[X]$. In \cite[Theorem 4.15]{AB}, it is also proved that $I[X]$ is a $2$-absorbing ideal of $R[X]$ iff $I$ is a $2$-absorbing ideal of $R$.

In this paper, we use content formula techniques to prove that their conjecture is true, i.e., $\omega_{R[X]}(I[X])=\omega_R(I)$ for an ideal $I$ of $R$, if one of the following conditions hold:

\begin{enumerate}

\item The ring $R$ is a Pr\"{u}fer domain.

\item The ring $R$ is a Gaussian ring such that its additive group is torsion-free.

\item The additive group of the ring $R$ is torsion-free and $I$ is a radical ideal of $R$.

\end{enumerate}

Since the content formula techniques for polynomials work for a generalization of these algebras known as content algebras, we recall the concept of content algebras, and then in the first section of this paper, we introduce Gaussian and Armendariz algebras and investigate them a bit. Finally in the second section, we prove that the formula $\omega_B(IB)=\omega_R(I)$ holds for some content algebras that are a generalization of their polynomial versions mentioned above.

Let $R$ be a commutative ring with identity and $B$ an $R$-algebra. For any element $f\in B$, the ideal $c(f) = \bigcap \lbrace I \colon I \text{~is an ideal of~} R \text{~and~} f \in IB \rbrace$ is attributed to it, called the content of $f$. Note that the content function $c$ is nothing but the generalization of the content of a polynomial $f\in R[X]$, which it is the ideal generated by its coefficients. The $R$-algebra $B$ is called a content $R$-algebra if the following conditions hold:

\begin{enumerate}
 \item
For all $f\in B$, $f\in c(f)B$.
 \item
(\textit{Faithful flatness}) $c(rf) = rc(f)$ For any $r \in R$ and $f \in B$, and $c(1_B) = R$.
 \item
(\textit{Dedekind-Mertens content formula}) For all $f,g$ in $B$, there exists a natural number $n$ such that $c(f)^n c(g) = c(f)^{n-1} c(fg)$.
\end{enumerate}

The algebra of all polynomials over an arbitrary ring in an arbitrary number of indeterminates and all semigroup rings whose semigroups are commutative, cancellative, and torsion-free are important and celebrated examples of content algebras (cf. \cite{OR} and \cite{No}). For more on content algebras and their examples, one may refer to \cite{OR}, \cite{R}, and \cite{ES2}, where content modules, content algebras and weak content algebras were introduced and investigated. On the other hand, the Dedekind-Mertens content formula and its generalization have been discussed in other papers like \cite{BG2}, \cite{G1}, \cite{GGP}, \cite{HH}, \cite{LR}, \cite{Na1}, \cite{NaY}, and \cite{P} with different perspectives as well.

 Now it is natural to ask when the simplest form of the Dedekind-Mertens content formula, i.e., $c(fg)=c(f)c(g)$, holds for all $f,g \in B$. It is obvious that if every nonzero finitely generated ideal of the ring $R$ is a cancelation ideal, i.e., $R$ is a Pr\"{u}fer domain, then from the Dedekind-Mertens content formula, we can deduce that $c(fg)=c(f)c(g)$ for all $f,g \in B$. We remind the reader that an ideal $I$ of a ring $R$ is called a cancellation ideal if for all ideals $J,K$ of $R$, $IJ = IK$ implies $J = K$. On the other hand, it is a celebrated result that if $D$ is a domain and $c(fg)=c(f)c(g)$ for all $f,g\in D[X]$, then $D$ is a Pr\"{u}fer domain (cf. \cite{G2} and \cite{T}).

 An arbitrary ring $R$ is called Gaussian if $c(fg)=c(f)c(g)$ for all $f,g \in R[X]$. There are many rings that are not domain, but still Gaussian. For more on Gaussian rings, one may refer to \cite{AC}, \cite{AG}, \cite{AK}, and \cite{BG1}. In the next section, we will define Gaussian algebras and discuss them. The importance of the first section is that it supplies many examples for what we prove in the second section on the Anderson-Badawi $\omega_{R[X]}(I[X])=\omega_R(I)$ conjecture.

 Throughout this paper, all rings are commutative with an identity different from zero. Also note that \emph{iff} always stands for ``if and only if".

\section{Gaussian and Armendariz algebras}

Let $B$ be an $R$-algebra such that $f\in c(f)B$ for all $f\in B$, where by $c(f)$, we mean the ideal $ \bigcap \lbrace I \colon I \text{~is an ideal of~} R \text{~and~} f \in IB \rbrace$. Let $f\in B$. Then $f\in c(f)B$, and this means that $f=\sum a_if_i$, where $a_i \in R$ and $f_i \in B$ and $c(f)=(a_1,a_2,\ldots, a_n)$. Similarly if $g\in B$, then $g=\sum b_jg_j$, where $b_j \in R$ and $g_j \in B$ and $c(g)=(b_1,b_2,\ldots, b_m)$. Then $fg= \sum a_ib_jf_ig_j \in c(f)c(g)B$, and hence $c(fg)\subseteq c(f)c(g)$ \cite[Proposition 1.1, p. 330]{R}. The question of when equality holds is the basis for the following definition:

\begin{definition}
 Let $B$ be an $R$-algebra such that $f\in c(f)B$ for all $f\in B$. We define $B$ to be a \textit{Gaussian $R$-algebra} if $c(fg) = c(f)c(g)$ for all $f,g \in B$.
\end{definition}

\begin{example}

Let $B$ be a content $R$-algebra such that $R$ is a Pr\"{u}fer domain. Since every nonzero finitely generated ideal of $R$ is a cancelation ideal of $R$, the Dedekind-Mertens content formula forces $B$ to be a Gaussian $R$-algebra.

\end{example}

Another example is given in the following remark.

\begin{remark}

\label{Gaussianexample1}

 Let $(R,\textbf{m})$ be a quasi-local ring with $\textbf{m}^2 = (0)$. If $B$ is a content $R$-algebra, then $B$ is a Gaussian $R$-algebra.

\begin{proof}
 Let $f,g \in B$ such that $c(f) \subseteq \textbf{m}$ and $c(g) \subseteq \textbf{m}$, then $c(fg) \subseteq c(f)c(g) \subseteq (0)$, so $c(fg) = c(f)c(g) = (0)$. Otherwise, one of them, say $c(f)$, is $R$ and according to the Dedekind-Mertens content formula, we have $c(fg) = c(g) = c(f)c(g)$.
\end{proof}

\end{remark}

Now we give another interesting class of Gaussian algebras. Recall that a ring $R$ is said to be a B\'{e}zout ring if every finitely generated ideal of $R$ is principal.

\begin{theorem}

\label{Gaussianexample2}

 Let $R$ be a B\'{e}zout ring and $S$ be a commutative, cancellative, torsion-free semigroup. Then $R[S]$ is a Gaussian $R$-algebra.

\end{theorem}

\begin{proof}
 Let $g = b_1X^{g_1}+b_2X^{g_2}+\cdots+b_nX^{g_n}$, where $b_i\in R$ and $g_i \in S$ for all $0\leq i \leq n$. Then there exists a $b \in R$, such that $c(g) = (b_1, b_2,\ldots,b_n) = (b)$. From this, we have $b_i = r_i b$ and $b = \sum s_i b_i$, where $r_i, s_i \in R$. Put $d = \sum s_ir_i$. Then $b = db$. Since $S$ is an infinite set, it is possible to choose $g_{n+1} \in S- \lbrace g_1,g_2,\ldots,g_n \rbrace $.

Let $g^ \prime = r_1X^{g_1}+r_2X^{g_2}+\cdots+r_nX^{g_n}+(1-d)X^{g_{n+1}}$. One can easily check that $g =g^\prime b$, $c(g^ \prime) =R$, and
$c(fg) = c(fg^ \prime b) = c(fg^ \prime)b =c(f)b = c(f)c(g)$ for all $f\in R[S]$.
\end{proof}

Note that the condition on the commutative semigroup $S$, i.e., being a cancellative and torsion-free semigroup, cannot be reduced \cite[Theorem 2]{Na2}.

Though the assertions of the following theorem are in the papers \cite{ES1} and \cite{AK}, we express them in the present paper$\text{'}$s terminology for the reader$\text{'}$s convenience.

\begin{theorem}

\label{Gaussianexample3}

Let $R$ be a ring. Then the following statements hold:

\begin{enumerate}

\item If $R$ is a Noetherian ring, then $R[[X]]$ is a content $R$-algebra.

\item If $R$ is a Dedekind domain, then $R[[X]]$ is a Gaussian $R$-algebra.

\end{enumerate}

\begin{proof}

(1): Let $R[[X]]$ be the ring of formal power series over the Noetherian ring $R$. For $f\in R[[X]]$, let $A_f$ denote the ideal of $R$ generated by the coefficients of $f$ (cf. \cite{F}). It is straightforward to see that $A_f \subseteq I$ iff $f\in I[[X]]$ for every ideal $I$ of $R$ and $f\in R[[X]]$. Since $R$ is a Noetherian ring, $I[[X]] = I\cdot R[[X]]$ for every ideal $I$ of $R$, and so $A_f \subseteq I$ iff $f\in I\cdot R[[X]]$. This implies that $A_f = c(f)$ (refer to the statement 1.2 in \cite{OR}). Obviously $c(rf) = A_{rf}= (r) A_f = rc(f)$ for any $r\in R$ and $f\in R[[X]]$ and $c(1_{R[[X]]})=R$. On the other hand, the Dedekind-Mertens formula holds for formal power series over Noetherian rings (\cite[Theorem 2.6]{ES1}). From the above, we deduce that $R[[X]]$ is a content $R$-algebra.

(2): If $R$ is a Dedekind domain, then every ideal of $R$ is a cancelation ideal, and therefore $c(fg) = c(f)c(g)$ for all $f,g\in R[[X]]$ (also refer to \cite[Theorem 2.4]{AK}) and $R[[X]]$ is a Gaussian $R$-algebra.
\end{proof}

\end{theorem}

We next define Armendariz algebras and show their relationship with Gaussian algebras. Armendariz rings were introduced in \cite{RC}. A ring $R$ is said to be an Armendariz ring if for all $f,g \in R[X]$ with $f = a_0+a_1X+\cdots+a_nX^n$ and $g = b_0+b_1X+\cdots+b_mX^m$, $fg = 0$ implies $a_ib_j = 0$ for all $0 \leq i \leq n$ and $0 \leq j \leq m$. This is equivalent to saying that if $fg = 0$, then $c(f)c(g) = 0$, and is our inspiration for defining Armendariz algebras.

\begin{definition}
 Let $B$ be an $R$-algebra such that $f\in c(f)B$ for all $f\in B$. We say $B$ is an \textit{Armendariz $R$-algebra} if $fg = 0$ implies $c(f)c(g) = (0) $ for all $f,g \in B$.
\end{definition}

An $R$-algebra $B$ is called a weak content algebra if $f\in c(f)B$ and $c(f)c(g) \subseteq \sqrt{c(fg)}$ for all $f,g \in B$ (\cite{R}). For example, if $B$ is a weak content $R$-algebra and $R$ is a reduced ring, then $B$ is an Armendariz $R$-algebra. This is because if $fg=0$, then $c(f)c(g) \subseteq \sqrt{c(fg)}=\sqrt{(0)}=(0)$.

\begin{theorem}
 Let $R$ be a ring, $(0)$ a $\textbf{p}$-primary ideal of $R$ such that $\textbf{p}^2 = (0)$, and $B$ a content $R$-algebra. Then $B$ is an Armendariz $R$-algebra.

\begin{proof}
 Let $f,g \in B$, where $fg = 0$. If $f = 0$ or $g = 0$, then $c(f)c(g) = 0$. Otherwise, suppose that $f \not= 0$ and $g \not= 0$. Therefore $f$ and $g$ are both zero-divisors of $B$. Since $(0)$ is a $\textbf{p}$-primary ideal of $R$, $(0)$ is a $\textbf{p}B$-primary ideal of $B$ \cite[p. 331]{R}, and therefore $\textbf{p}B$ is the set of zero-divisors of $B$. So $f,g \in \textbf{p}B$, and this means that $c(f) \subseteq \textbf{p}$ and $c(g) \subseteq \textbf{p}$. Finally, $c(f)c(g) \subseteq \textbf{p}^2 = (0)$.
\end{proof}

\end{theorem}

In order to characterize Gaussian algebras in terms of Armendariz algebras, we mention the following useful lemma.

\begin{lemma}
 Let $R$ be a ring and $I$ an ideal of $R$. If $B$ is a Gaussian $R$-algebra, then $B/IB$ is a Gaussian $(R/I)$-algebra.

 \begin{proof}
 Straightforward.
 \end{proof}

\end{lemma}

\begin{theorem}
 Let $B$ be a content $R$-algebra. Then $B$ is a Gaussian $R$-algebra iff $B/IB$ is an Armendariz $(R/I)$-algebra for every ideal $I$ of $R$.

\begin{proof}
 $ ( \Rightarrow ): $ According to the above lemma, since $B$ is a Gaussian $R$-algebra, $B/IB$ is a Gaussian $(R/I)$-algebra. On the other hand, any Gaussian algebra is an Armendariz algebra and this completes the proof.

$ ( \Leftarrow ): $ In the beginning of this section, we proved that if $B$ is an $R$-algebra such that $f\in c(f)B$ for all $f\in B$, then  $c(fg) \subseteq c(f)c(g)$ for all $f,g \in B$ \cite[Proposition 1.1, p. 330]{R}. Therefore, we need to prove that $c(f)c(g) \subseteq c(fg)$. Put $I = c(fg)$. Since $B/IB$ is an Armendariz $(R/I)$-algebra and $c(fg+IB) = I$, we have $c(f+IB)c(g+IB) = I$, and this means that $c(f)c(g) \subseteq c(fg)$.
\end{proof}

\end{theorem}

The two recent theorems are generalizations of the similar theorems for polynomial rings in \cite{AC}.

After this short introductory section on Gaussian algebras, we pass to the next section to discuss the Anderson-Badawi $\omega_{R[X]}(I[X])=\omega_R(I)$ conjecture.

\section{Anderson-Badawi $\omega_{R[X]}(I[X])=\omega_R(I)$ conjecture}

The concept of $2$-absorbing ideals was introduced and investigated in \cite{B}. This concept has been generalized for any positive integer $n$ by Anderson and Badawi. In their paper \cite{AB}, a proper ideal $I$ of a commutative ring $R$ is defined as an $n$-absorbing ideal of $R$ if whenever $x_1 \cdots x_{n+1} \in I$ for $x_1, \ldots, x_{n+1} \in R$, then there are $n$ of the $x_i$'s whose product is in $I$. In the final section of the paper \cite{AB}, the authors define a strongly $n$-absorbing ideal of a ring as follows: A proper ideal $I$ of a commutative ring $R$ is called a strongly $n$-absorbing ideal if whenever $I_1 \cdots I_{n+1} \subseteq I$ for ideals $I_1, \ldots, I_{n+1}$ of $R$, then there are $n$ of the $I_i$'s whose product is contained in $I$.

Clearly a $1$-absorbing ideal is just a prime ideal, and it is a famous result in commutative ring theory that $I$ is a prime ideal of $R$ iff $I[X]$ is a prime ideal of $R[X]$. In \cite[Theorem 4.15]{AB}, it is also proved that $I[X]$ is a $2$-absorbing ideal of $R[X]$ iff $I$ is a $2$-absorbing ideal of $R$. One can easily check that if an ideal is a strongly $n$-absorbing ideal of $R$, then it is an $n$-absorbing ideal of $R$, and Anderson and Badawi in \cite{AB} conjectured that these two concepts are equivalent and they showed that the two concepts are equivalent for Pr\"{u}fer domains \cite[Corollary 6.9]{AB}.

In the same paper, Anderson and Badawi also conjecture that $\omega_{R[X]}(I[X])=\omega_R(I)$ for any ideal $I$ of an arbitrary ring $R$, where $\omega_R(I)= \min \{n \colon \text{ $I$ is an $n$-absorbing ideal of $R$}\}$. In the following, we prove that Anderson-Badawi $\omega_{R[X]}(I[X])=\omega_R(I)$ conjecture holds for Pr\"{u}fer domains. Actually, we prove a generalization of this formula for content algebras over Pr\"{u}fer domains.

\begin{theorem}

\label{ABconjecture}

Let $R$ be a Pr\"{u}fer domain, $I$ an ideal of $R$, and $B$ a content $R$-algebra. Then $\omega_B(IB) = \omega_R(I)$.

\begin{proof}
Let $B$ be a content $R$-algebra. Then it is easy to see that $R$ can be considered as a subring of $B$. This means that if $I$ is an ideal of $R$, then $\omega_R(IB \cap R) \le \omega_B(IB)$ by \cite[Corollary 4.3]{AB}. But $IB \cap R = I$ for any ideal $I$ of $R$, since $B$ is a content $R$-algebra. Therefore $\omega_R(I) \le \omega_B(IB)$.

It is obvious that $\omega_R(I)=0$ iff $\omega_B(IB) = 0$, since $\omega_R(I)=0$ iff $I=R$ for any ideal $I$ of $R$, according to its definition in \cite{AB}. Also note that in content algebras, $IB=B$ iff $I=R$. Now let $\omega_R(I)=n$ for a positive integer $n$. We claim that $IB$ is an $n$-absorbing ideal of $B$. Since $R$ is a Pr\"{u}fer domain and $B$ is a content $R$-algebra, $B$ is a Gaussian $R$-algebra. Now assume that $f_1 \cdots f_{n+1} \in IB$ for arbitrary $f_1, \ldots, f_{n+1} \in B$. It is clear that $c(f_1 \cdots f_{n+1}) \subseteq I$. But $B$ is a Gaussian $R$-algebra, so $c(f_1 \cdots f_{n+1})=c(f_1) \cdots c(f_{n+1})$. On the other hand, by \cite[Corollary 6.9]{AB}, $I$ is a strongly $n$-absorbing ideal of $R$ and this implies $c(f_1) \cdots c(f_{i-1})c(f_{i+1}) \cdots c(f_{n+1}) \subseteq I$ for some $i$ with $1 \leq i \leq n+1$. Therefore, $c(f_1 \cdots f_{i-1} f_{i+1} \cdots f_{n+1}) \subseteq I$, and finally $f_1 \cdots f_{i-1} f_{i+1} \cdots f_{n+1} \in IB$. So we have already proved that $n=\omega_R(I) \le \omega_B(IB) \le n$. Now let $\omega_B(IB)=n$ for a positive integer $n$. First we prove that $I$ is an $n$-absorbing ideal of $R$. To show that, we let $a_1 \cdots a_{n+1} \in I$. Then $a_1 \cdots a_{n+1} \in IB$, since $I\subseteq IB$. But $IB$ is an $n$-absorbing ideal of $B$ and therefore $a_1 \cdots a_{i-1}a_{i+1} \cdots a_{n+1} \in IB$ for some $i$ with $1 \leq i \leq n+1$. Since $IB \cap R =I$ and $a_1 \cdots a_{i-1}a_{i+1} \cdots a_{n+1} \in R$, we have $a_1 \cdots a_{i-1}a_{i+1} \cdots a_{n+1} \in I$. This means that $\omega_R(I)$ is finite and nonzero. So we let $\omega_R(I)=m$ be a positive integer and according to what we proved in above, we have $n=\omega_B(IB) = \omega_R(I)=m$. From what we said, we conclude that $\omega_B(IB) = \infty$ iff $\omega_R(I) = \infty$, and the proof is complete.
\end{proof}

\end{theorem}

\begin{corollary}

Let $R$ be a domain. Then the following statements hold.

\begin{enumerate}

\item If $R$ is a Pr\"{u}fer domain, then $\omega_{R[X]}(I[X]) = \omega_R(I)$ for every ideal $I$ of $R$.

\item If $R$ is a Dedekind domain, then $\omega_{R[[X]]}(I[[X]]) = \omega_R(I)$ for every ideal $I$ of $R$.

\end{enumerate}

\end{corollary}

\begin{remark}

\label{ABconjecture2}

Recall that Anderson and Badawi conjectured that the two concepts of $n$-absorbing ideal and strongly $n$-absorbing ideal are equivalent (\cite[Conjecture 1]{AB}). In \cite{DP}, A. Y. Darani and E. R. Puczy{\l}owski show that this conjecture holds for rings whose additive group is torsion-free. On the other hand, if $R$ is a ring such that every $n$-absorbing ideal of $R$ is strongly $n$-absorbing and $B$ is a faithfully flat Gaussian $R$-algebra, then a proof similar to the proof of Theorem \ref{ABconjecture} shows that $\omega_B(IB) = \omega_R(I)$. So we have the following result:

\begin{corollary}

If $R$ is a Gaussian ring and its additive group is torsion-free, then $\omega_{R[X]}(I[X])=\omega_R(I)$ for every ideal $I$ of $R$.

\end{corollary}

\end{remark}

\begin{example} In Theorem \ref{ABconjecture}, we proved that if $R$ is a Pr\"{u}fer domain, then $\omega_{R[X]}(I[X])=\omega_R(I)$ for every ideal $I$ of $R$. In the following, we give an example of a ring $S$ satisfying $\omega_{S[X]}(I[X])=\omega_S(I)$, while the ring $S$ is not a domain. Let $k$ be a field with characteristic 0 and put $R=k[[X_1,\ldots,X_n]]$. Then $R$ is a local ring with the maximal ideal $\textbf{m}=(X_1,\ldots,X_n)$. We consider the ring $S=R/\textbf{m}^2$. It is easy to check that $(S,\textbf{n})$ is a local ring with $\textbf{n}^2=(0)$, where $\textbf{n}=\textbf{m}/\textbf{m}^2$. Therefore according to Remark \ref{Gaussianexample1}, $S$ is a Gaussian ring. On the other hand, since the characteristic of the field $k$ is 0, the additive group of the ring $S$ is torsion-free, and finally $\omega_{S[X]}(I[X]) = \omega_S(I)$ for every ideal $I$ of $S$, while $S$ is not a domain.

\end{example}

\begin{theorem}

Let $B$ be a content $R$-algebra and $R$ be a ring such that every $n$-absorbing ideal of $R$ is a strongly $n$-absorbing ideal of $R$ for any positive integer $n$ (for example, let the additive group of $R$ be torsion free (\cite[Theorem 4.2]{DP})). If $I$ is a radical ideal of $R$, then $\omega_B(IB) = \omega_R(I)$.

\begin{proof}
We just need to prove that if $\omega_R(I)=n$ for a positive integer $n$, then $IB$ is an $n$-absorbing ideal of $B$, since the rest of the proof is similar to the proof of Theorem \ref{ABconjecture}. So let $f_1 \cdots f_{n+1} \in IB$. Obviously $c(f_1 \cdots f_{n+1}) \subseteq I$. Let $g=f_2 \cdots f_{n+1}$. By the Dedekind-Mertens content formula for content algebras, there is a natural number $l_1$ such that $c(f_1)^{l_1}c(g) = c(f_1)^{{l_1}-1}c(f_1 g)$ and since $c(f_1 g) \subseteq I$, we have $c(f_1)^{l_1}c(g) \subseteq I$. Continuing this process, we get the natural numbers $l_2, \ldots, l_{n}$ such that $c(f_1)^{l_1} \cdots c(f_{n})^{l_{n}}c(f_{n+1}) \subseteq I$. Obviously, if we let $l=\max\{l_1, \cdots, \l_n\}$, then $(c(f_1) \cdots c(f_{n+1}))^l \subseteq I$, and since $I=\sqrt{I}$, we have $c(f_1) \cdots c(f_{n+1}) \subseteq I$.

But $\omega_R(I)=n$. So $I$ is an $n$-absorbing ideal and according to our assumptions, a strongly $n$-absorbing ideal of $R$. Thus $c(f_1) \cdots c(f_{i-1})c(f_{i+1}) \cdots c(f_{n+1}) \subseteq I$ for some $i$ with $1 \leq i \leq n+1$.

On the other hand, $c(f_1 \cdots f_{i-1} f_{i+1} \cdots f_{n+1}) \subseteq c(f_1) \cdots c(f_{i-1})c(f_{i+1}) \cdots c(f_{n+1})$, hence $f_1 \cdots f_{i-1} f_{i+1} \cdots f_{n+1} \in IB$.
\end{proof}

\end{theorem}

\begin{corollary}

Let $R$ be a ring and $I$ a radical ideal of $R$. Then the following statements hold.

\begin{enumerate}

\item If the additive group of the ring $R$ is torsion-free, then $\omega_{R[X]}(I[X])=\omega_R(I)$.

\item If $R$ is a Noetherian ring and the additive group of the ring $R$ is torsion-free, then $\omega_{R[[X]]}(I[[X]])=\omega_R(I)$.

\end{enumerate}

\end{corollary}

\section{Acknowledgments} The author is partly supported by the University of Tehran and wishes to thank Prof. Dara Moazzami for his encouragements.

\bibliographystyle{plain}

\begin{thebibliography}{15.}

\bibitem {AB} D. F. Anderson and A. Badawi, {\em On $n$-absorbing ideals of commutative rings}, Comm. Algebra. {\bf 39} (2011), 1646--1672.

\bibitem {AC} D. D. Anderson and V. Camillo, {\em Armendariz rings and Gaussian rings}, Comm. Algebra. {\bf 26} (1998), 2265--2272.

\bibitem {AG} J. T. Arnold and R. Gilmer, {\em On the content of polynomials}, Proc. Amer. Math. Soc. {\bf 40} (1970), 556--562.

\bibitem {AK} D. D. Anderson and B. G. Kang, {\em Content formulas for polynomials and power series and complete integral closure}, J. Algebra, {\bf 181} (1996), 82--94.

\bibitem {B} A. Badawi, {\em On $2$-absorbing ideals of commutative rings}, Bull. Australl. Math. Soc. {\bf 75} (2007), 417--429.

\bibitem {BG1} S. Bazzoni and S. Glaz, {\em Gaussian properties of total rings of quotients}, J. Algebra {\bf310} (2007), no. 1, 180--193.

\bibitem {BG2} W. Bruns and A. Guerrieri, {\em The Dedekind-Mertens formula and determinantal rings}, Proc.
Amer. Math. Soc. {\bf127} (1999), no. 3, 657--663.

\bibitem {DP}  A. Y. Darani and E. R. Puczy{\l}owski, {\em On $2$-absorbing commutative semigroups and their applications to rings}, Semigroup Forum {\bf 86} (2013), 83--91.

\bibitem {ES1} N. Epstein and J. Shapiro, {\em A Dedekind-Mertens theorem for power series rings}, to appear in Proc. Amer. Math. Soc.

\bibitem {ES2} P. Eakin and J. Silver, {\em Rings which are almost polynomial rings}, Trans. Amer. Math. Soc. {\bf 174} (1974), 425--449.

\bibitem {F} D. E. Fields, {\em Zero divisors and nilpotent elements in power series rings}, Proc. Amer. Math. Soc. {\bf 27} (3) (1971), 427--433.

\bibitem {G1} R. Gilmer, {\em Multiplicative Ideal Theory}, Marcel Dekker, New York, 1972.

\bibitem {G2} R. Gilmer, {\em Some applications of the Hilfssatz von Dedekind-Mertens}, Math. Scand. {\bf 20} (1967), 240--244.

\bibitem {GGP} R. Gilmer, A. Grams and T. Parker, {\em Zero divisors in power series rings}, J. Reine Angew. Math.  {\bf278} (1975), 145--164.

\bibitem {HH} W. Heinzer and C. Huneke, {\em The Dedekind-Mertens Lemma and the content of polynomials}, Proc.
Amer. Math. Soc. {\bf126} (1998), 1305--1309.

\bibitem {LR} K. A. Loper and M. Roitman, {\em The content of a Gaussian polynomial is invertible}, Proc. Amer. Math. Soc. {\bf 133} (2005), 1267--1271.

\bibitem {Na1} P. Nasehpour, {\em Zero-divisors of content algebras}, Arch. Math. (Brno), {\bf 46} (4) (2010), 237--249.

\bibitem {Na2} P. Nasehpour, {\em Zero-divisors of semigroup modules}, Kyungpook Math. J., {\bf 51} (1) (2011), 37--42.

\bibitem {NaY} P. Nasehpour and S. Yassemi, {\em $M$-cancellation ideals}, Kyungpook Math. J., {\bf 40} (2000), 259--263.

\bibitem {No} D. G. Northcott, {\em A generalization of a theorem on the content of polynomials}, Proc. Cambridge Phil. Soc. {\bf 55} (1959), 282--288.

\bibitem {OR} J. Ohm and D. E. Rush, {\em Content modules and algebras}, Math. Scand. {\bf 31} (1972), 49--68.

\bibitem {P} H. Pr\"{u}fer, {\em Untersuchungen \"{u}ber Teilbarkeitseigenschaften in K\"{o}rpern}, J. Reine Angew. Math. {\bf 168} (1932), 1--36.

\bibitem {RC} M. B. Rege and S. Chhawchharia, {\em Armendariz rings}, Proc. Japan Acad. Ser. A Math. Sci. {\bf 73}, Number 1 (1997), 14--17.

\bibitem {R} D. E. Rush, {\em Content algebras}, Canad. Math. Bull. Vol. {\bf 21} (3) (1978), 329--334.

\bibitem {T} H. Tsang, {\em Gauss' Lemma}, dissertation, University of Chicago, Chicago, 1965.



\end{thebibliography}

\end{document}